\documentclass[12pt, reqno, a4paper]{amsart}

\usepackage{amssymb}
\usepackage{hyperref}

\usepackage[english]{babel}

\theoremstyle{plain}

\newtheorem*{corollary}{Corollary}
\newtheorem{lemma}{Lemma}
\newtheorem{theorem}{Theorem}
\newtheorem*{conjecture}{Conjecture}

\theoremstyle{remark}
\newtheorem*{remark}{Remark}

\theoremstyle{definition}

\newtheorem{example}{Example}

\DeclareMathOperator{\Id}{Id}

\DeclareMathOperator{\id}{id}
\DeclareMathOperator{\ch}{char}
\DeclareMathOperator{\im}{im}
\DeclareMathOperator{\ad}{ad}
\DeclareMathOperator{\GL}{GL}

\DeclareMathOperator{\End}{End}
\DeclareMathOperator{\Aut}{Aut}

\DeclareMathOperator{\sign}{sign}

\DeclareMathOperator{\Ann}{Ann}
\DeclareMathOperator{\Hom}{Hom}
\DeclareMathOperator{\PIexp}{PIexp}

\begin{document}

\title[Derivations, gradings, actions of algebraic groups]{Derivations, gradings, actions of algebraic groups, and codimension growth of polynomial identities}

\author{A.\,S.~Gordienko}
\address{Vrije Universiteit Brussel, Belgium}
\email{alexey.gordienko@vub.ac.be} 
\author{M.\,V.~Kochetov}
\address{Memorial University of Newfoundland, St. John's, NL, Canada}
\email{mikhail@mun.ca}

\begin{abstract}
Suppose a finite dimensional semisimple Lie algebra~$\mathfrak g$ acts by derivations on 
a finite dimensional associative or Lie algebra $A$ over a field of characteristic $0$.
We prove the $\mathfrak g$-invariant analogs of Wedderburn~--- Mal'cev and Levi theorems,
and the analog of Amitsur's conjecture
on asymptotic behavior for codimensions of polynomial identities with derivations
of $A$.  It turns out that for associative algebras the differential PI-exponent coincides
with the ordinary one.
Also we prove the analog of Amitsur's conjecture
for finite dimensional associative algebras with an action of a reductive affine algebraic group
by automorphisms and anti-automorphisms or graded by an arbitrary Abelian group.
In addition, we provide criteria for $G$-, $H$- and graded simplicity in terms of codimensions.
\end{abstract}


\thanks{
The first author is supported by FWO Vlaanderen Pegasus Marie Curie post doctoral fellowship
(Belgium). The second author is supported by by the Natural Sciences and Engineering 
Research Council (NSERC) of Canada, Discovery Grant \# 341792-07.
}

\maketitle

\section{Introduction} 

The Levi Theorem is one of the main results of structure Lie theory, as
well as the Wedderburn~--- Mal'cev Theorem is one of the central results
in structure ring theory. We are interested in Lie and associative algebras
with an additional structure, e.g. graded, $H$-(co)module, or $G$-algebras, and in decompositions
compatible with these structures. In 1957, E.J.~Taft proved~\cite{Taft} the $G$-invariant Levi and Wedderburn~--- Mal'cev
theorems 
for $G$-algebras with an action of a finite group $G$ by automorphisms and anti-automorphisms.
Due to a well-known duality between $G$-gradings and $G$-actions, Taft's result implies
graded decompositions of algebras graded by a finite Abelian group $G$
over an algebraically closed field of characteristic $0$. 
The study of Wedderburn decompositions for $H$-module algebras was started by A.\,V. Sidorov~\cite{Sidorov}
in 1986. In 1999, D. \c Stefan and F. Van Oystaeyen~\cite{SteVanOyst} proved
 the $H$-coinvariant Wedderburn~--- Mal'cev Theorem for finite dimensional $H$-comodule associative algebras,
 where $H$ is
a Hopf algebra  with an $\ad$-invariant left integral $t\in H^*$
such that $t(1)=1$.
In particular, they proved the $H$-(co)invariant Wedderburn~--- Mal'cev Theorem for finite dimensional semisimple $H$
over a field of characteristic $0$, the graded Wedderburn~--- Mal'cev Theorem for  any grading group provided that the Jacobson radical is graded too,
and the $G$-invariant Wedderburn~--- Mal'cev Theorem for associative algebras with a rational action of a reductive algebraic group $G$ by automorphisms only.
 The graded Levi Theorem for finite dimensional Lie algebras over an algebraically closed field of characteristic $0$, graded by a finite group, was proved by D.~Pagon, D.~Repov\v s, and M.V.~Zaicev~\cite{PaReZai} in 2011.

In 2012, the first author proved~\cite{ASGordienko4} the $H$-coinvariant Levi Theorem in the case when the Hopf algebra $H$ has an $\ad$-invariant left integral $t\in H^*$
such that $t(1)=1$. As a consequence he obtained the $H$-invariant Levi Theorem
for $H$-module Lie algebras for a finite dimensional semisimple Hopf algebra $H$,
 the graded Levi Theorem for an arbitrary grading group,
and the $G$-invariant Levi Theorem for Lie algebras with a rational action of a reductive algebraic group $G$
by automorphisms only.

In this paper we prove the $G$-invariant Wedderburn~--- Mal'cev and Levi theorems (Theorems~\ref{TheoremAffAlgGrWedderburn} and~\ref{TheoremAffAlgGrLevi} in Subsection~\ref{SubsectionGDecomp})
where $G$ is a reductive affine algebraic group over an algebraically closed field of characteristic $0$, acting rationally by automorphisms and anti-automorphisms. Also we prove the $\mathfrak g$-invariant Wedderburn~--- Mal'cev and Levi theorems (Theorems~\ref{TheoremDerWedderburn} and~\ref{TheoremDerLevi} in Subsection~\ref{SubsectionGApplToDiff})
where $\mathfrak g$ is a finite dimensional semisimple Lie algebra over
a field of characteristic $0$, acting by derivations.

One of the applications of invariant decompositions is in the combinatorial theory of graded, differential, $G$- or $H$-polynomial 
identities.

In the 1980's, a conjecture about the asymptotic behaviour
of codimensions of ordinary polynomial identities was made
by S.A.~Amitsur. Amitsur's conjecture was proved in 1999 by
A.~Giambruno and M.V.~Zaicev~\cite[Theorem~6.5.2]{ZaiGia} for associative algebras, in 2002 by M.V.~Zaicev~\cite{ZaiLie}
 for finite dimensional Lie algebras, and in 2011 by A.~Giambruno,
 I.P.~Shestakov, M.V. Zaicev for finite dimensional Jordan and alternative
 algebras~\cite{GiaSheZai}. In 2011, the first author proved its analog
 for polynomial identities of finite dimensional representations of Lie
 algebras~\cite{ASGordienko}.
 
 Alongside with ordinary polynomial
identities of algebras, graded, differential, $G$- and
$H$-identities are
important too~\cite{BahtGiaZai, BahturinLinchenko, BahtZaiGraded, BahtZaiGradedExp, BahtZaiSehgal, 
 BereleHopf, Chuang, DeFilippis, Kharchenko, Linchenko, Sharma}.
 Usually, to find such identities is easier
  than to find the ordinary ones. Furthermore, each of these types of identities completely determines  the ordinary polynomial identities. 
Therefore the question arises whether the conjecture
holds for graded codimensions, $G$-, $H$-codimensions,
and codimensions of polynomial identities with derivations.
The analog of Amitsur's conjecture
for codimensions of graded identities was proved in 2010--2011 by
E.~Aljadeff,  A.~Giambruno, and D.~La~Mattina~\cite{AljaGia, AljaGiaLa, GiaLa}
  for all associative PI-algebras graded by a finite group.
   As a consequence, they proved the analog  of the conjecture for $G$-codimensions
   for any associative PI-algebra with an action of a finite Abelian group $G$ by automorphisms.
The case when $G=\mathbb Z_2$
acts on a finite dimensional associative algebra by automorphisms and anti-automorphisms
(i.e. polynomial identities with involution)
 was considered by A.~Giambruno and
 M.V.~Zaicev~\cite[Theorem~10.8.4]{ZaiGia}
 in 1999.
 
  In 2012, the first author~\cite{ASGordienko3} proved
 the analog of Amitsur's conjecture for polynomial
 $H$-identities of finite dimensional associative algebras
 with a generalized $H$-action under some assumptions on the $H$-action. As a consequence, 
 the analog of Amitsur's conjecture was proved for
 $G$-codimensions of finite dimensional associative algebras with
 an action of an arbitrary finite group $G$ by automorphisms and anti-automorphisms,
 and for $H$-codimensions of finite dimensional $H$-module associative algebras for a
 finite dimensional semisimple Hopf algebra~$H$.
     
      In 2012, the first author~\cite{ASGordienko5} proved 
  the analog of Amitsur's conjecture for graded polynomial identities
  of finite dimensional
Lie algebras graded by any group, for $G$-identities
of finite dimensional
Lie algebras with a rational action of a reductive affine algebraic group,
and for $H$-identities of finite dimensional $H$-module Lie algebras under some assumptions on the $H$-action.
(A particular case of this was proved in~\cite{ASGordienko2}.)

This article is concerned with the analog of Amitsur's conjecture for codimensions of differential identities of finite dimensional Lie
  and associative algebras with an action of a finite dimensional semisimple Lie algebra
  by derivations
(Section~\ref{SectionDerH}),
$G$-codimensions of associative algebras with a rational action of a reductive affine algebraic group
$G$ by automorphisms and anti-automorphisms (Section~\ref{SectionG}), and
graded codimensions of associative algebras graded by an arbitrary Abelian group (Section~\ref{SectionGraded}).
Here we use an easy trick (see Theorem~\ref{TheoremAssGenHopf})
in order to remove in~\cite[Theorem~5]{ASGordienko3} the requirement for $\dim H$ to be finite.

 In Section~\ref{SectionPIExp} we provide explicit formulas
for the exponents of differential, graded, and $G$-identities that are natural generalizations of the formulas
for the ordinary PI-exponents (see~\cite[Section~6.2]{ZaiGia} and~\cite[Definition~2]{ZaiLie}). 
It turns out that the differential PI-exponent of a finite dimensional associative algebra coincides with the ordinary one if the Lie algebra acting by derivations is finite dimensional semisimple. The same is true for the exponent
of $G$-identities when $G$ is a connected reductive affine algebraic group.
In Section~\ref{SectionExamples} we provide criteria for graded, $G$-,
and $H$-simplicity; in the proof, we will use an upper bound for codimensions, which is established in Section~\ref{SectionSncocharAndUpper}.


\section{Structure theory}
\subsection{Wedderburn~--- Mal'cev and Levi decompositions for $G$-algebras}
\label{SubsectionGDecomp}

We use the exponential notation for the action of a
 group.
Let $A$ be an algebra over a field $F$.
Recall that $\psi \in \GL(A)$ is an {\itshape
automorphism} of $A$ if $(ab)^\psi = a^\psi b^\psi$
for all $a,b \in A$
and {\itshape
anti-automorphism} of $A$ if $(ab)^\psi = b^\psi a^\psi$
for all $a, b \in A$. The automorphisms of $A$
form a group, which is denoted by $\Aut(A)$.
The automorphisms and anti-automorphisms of $A$
form a group, which is denoted by $\Aut^{*}(A)$.
Note that $\Aut(A)$ is a normal subgroup of $\Aut^{*}(A)$
of index ${}\leqslant 2$.

Let $G$ be a group.
  We say that an associative algebra $A$ is an \textit{algebra with $G$-action}
  or a \textit{$G$-algebra}
   if $A$ is endowed with a homomorphism $\varphi \colon G \to
  \Aut^{*}(A)$.
Note that $G_0 :=  \varphi^{-1}(\Aut(A))$ is
a normal subgroup of $G$ of index ${}\leqslant 2$.

We claim that the following theorem holds:

\begin{theorem}\label{TheoremAffAlgGrWedderburn}
Let $A$ be a finite dimensional associative algebra over an algebraically closed field $F$ of characteristic $0$ and let $G$ be a reductive affine algebraic group over $F$.
Suppose $A$ is endowed with a rational action of $G$ by automorphisms and anti-automorphisms.
 Then there exists a maximal semisimple subalgebra $B\subseteq A$ such that
$A=B\oplus J$ (direct sum of $G$-invariant subspaces) where $J:=J(A)$ is the Jacobson radical of $A$.
\end{theorem}
\begin{proof}
First we prove the theorem for the case $J^2=0$.

If $G$ is acting by automorphisms only, then the theorem follows from~\cite[Corollary~2.10]{SteVanOyst}.
Hence we may assume that the subgroup $G_0 \subset G$ is of index $2$.
Note that $G_0$ is closed since it is defined by polynomial equations.

Moreover, $J$ is $G$-invariant since the maximal nilpotent ideal is invariant under all
automorphisms and anti-automorphisms. Let $\pi \colon A \to A/J$ be the corresponding
natural projection. By~\cite[Corollary~2.10]{SteVanOyst}, there exists a $G_0$-equivariant
homomorphic embedding $\varphi \colon  A/J \hookrightarrow A$ such that $\pi\varphi=\id_{A/J}$.

Fix $g \in G\backslash G_0$. Define the map $\tilde\varphi \colon  A/J \to A$
by $\tilde \varphi(a) = (\varphi(a)+g\varphi(g^{-1} a))/2$ for $a \in A/J$.
Then $$\tilde \varphi(ha)=(\varphi(ha)+g\varphi(g^{-1} ha))/2
= (h\varphi(a)+g\varphi((g^{-1} hg)g^{-1}a))/2
=$$ $$(h\varphi(a)+g(g^{-1} hg)\varphi(g^{-1}a))/2=h\tilde \varphi(a)
\text{ for all }h\in G_0,\ a\in A/J$$
and \begin{equation*}\begin{split}\tilde \varphi(g a)=(\varphi(g a)+g\varphi(a))/2= 
g(g^{-1}\varphi(g^2 g^{-1} a)+\varphi(a))/2= \\
g(g^{-1}g^2\varphi(g^{-1} a)+\varphi(a))/2 = g \tilde\varphi(a)
\text{ for all }a\in A/J.\end{split}\end{equation*}
Hence $\tilde\varphi$ is $G$-equivariant.

 Let $a \in A/J$. Then $\pi\tilde\varphi(a)=
 (\pi\varphi(a)+g\pi\varphi(g^{-1} a))/2=a$. We claim that $\tilde\varphi$
 is a homomorphism of algebras. 
 
 First we observe that a linear map $\psi \colon A/J \to A$,
 such that $\pi\psi = \id_{A/J}$, is a homomorphism of algebras if and only if $(\varphi-\psi)
 \colon A/J \to J$ is a {\itshape $(\varphi,\varphi)$-skew derivation}, i.e.
 $(\varphi-\psi)(ab)=(\varphi-\psi)(a)\varphi(b)+\varphi(a)(\varphi-\psi)(b)$
 for all $a,b \in A/J$.  Indeed, if $\psi \colon A/J \to A$
 is a homomorphism of algebras, then  $$(\varphi-\psi)(ab)=\varphi(a)\varphi(b)-\psi(a)\psi(b)
 =(\varphi-\psi)(a)\varphi(b)+\psi(a)(\varphi-\psi)(b) =$$
 $$(\varphi-\psi)(a)\varphi(b)+\psi(a)(\varphi-\psi)(b)+(\varphi-\psi)(a)(\varphi-\psi)(b)
 = (\varphi-\psi)(a)\varphi(b)+\varphi(a)(\varphi-\psi)(b)$$
 since $(\varphi-\psi)(a)(\varphi-\psi)(b) \in J^2 = 0$ for all $a,b\in A/J$.
 The converse is proved by a similar calculation.
 
 Hence $a \mapsto (\varphi(a) - g \varphi(g^{-1}a))$, $a\in A/J$, is a $(\varphi,\varphi)$-skew derivation, and
 $$a \mapsto \varphi(a)-(\varphi(a) - g \varphi(g^{-1}a))/2=\tilde\varphi(a)$$
 is a homomorphism of algebras. Therefore, we can take $B = \im\tilde\varphi$,
 $A=\im\tilde\varphi \oplus \ker\tilde\varphi=B\oplus J$, and the theorem is proved for
 the case $J^2=0$.
 
 We prove the general case by induction on $\dim J$. Suppose $J^2\ne 0$. Hence
 $\dim (J/J^2) < \dim J$ and,
  by induction, $A/J^2 = A_1/J^2 \oplus J/J^2$ where $A_1 \subseteq A$ is a $G$-invariant subalgebra
  such that $A_1/J^2 \cong A/J$ is semisimple. Since the Jacobson radical is nilpotent,  $\dim J^2 < \dim J$ and, by induction, $A_1 = B \oplus J^2$ where $B \cong A/J$ is a $G$-invariant semisimple subalgebra.
  Now we notice that $A=B\oplus J$ (direct sum of $G$-invariant subspaces).
\end{proof}

Analogously, we derive Theorem~\ref{TheoremAffAlgGrLevi} from~\cite[Theorem~5]{ASGordienko4}.

\begin{theorem}\label{TheoremAffAlgGrLevi}
Let $L$ be a finite dimensional Lie algebra over an algebraically closed field $F$ of characteristic $0$
and let $G$ be a reductive affine algebraic group over $F$.
Suppose $L$ is endowed with a rational action of $G$ by automorphisms and anti-automorphisms.
 Then there exists a maximal semisimple subalgebra $B$ in $L$ such that
$L=B\oplus R$ (direct sum of $G$-invariant subspaces).
\end{theorem}

\subsection{Connection between derivations and automorphisms}

The main trick in our investigation of algebras with derivations is to replace the action of a Lie algebra by derivations with an action
of an affine algebraic group by automorphisms, which in our situation has been studied better.

\begin{theorem}\label{TheoremDerAutConnection}
Let $A$ be a finite dimensional algebra, 
not necessarily associative, over an algebraically closed field $F$ of characteristic $0$. Suppose a finite dimensional semisimple Lie algebra $\mathfrak g$
is acting on $A$ by derivations. Then there exists a rational
representation of a simply connected semisimple affine algebraic group $G$ on $A$ by automorphisms
such that \begin{enumerate}\item the Lie algebra of $G$ equals $\mathfrak g$;
\item the $\mathfrak g$-action on $A$ is the differential of the $G$-action on $A$;
\item all $\mathfrak g$-submodules
in $A$ are $G$-invariant subspaces and vice versa.\end{enumerate}
\end{theorem}
\begin{proof}
By~\cite[Chapter~XVIII, Theorem~5.1]{Hochschild}, there exists a simply connected affine algebraic
group $G$ such that the Lie algebra of $G$ is isomorphic to $\mathfrak g$.
The $\mathfrak g$-module $A$ is the direct sum of irreducible $\mathfrak g$-submodules
that correspond to some dominant weights of $\mathfrak g$.
We define on the irreducible $\mathfrak g$-submodules the rational action of $G$ corresponding to those weights.

We claim that $G$ acts on $A$ by automorphisms. Indeed, we can treat the multiplication $\mu \colon A \otimes A \to A$
as an element $\mu = \sum_{i} \mu_{1i} \otimes \mu_{2i} \otimes \mu_{3i} \in A^*\otimes A^* \otimes A$. We have the following action of $G$ and $\mathfrak g$
on the space $A^*\otimes A^* \otimes A$: $$g(u(\cdot) \otimes v(\cdot) \otimes w)=u(g^{-1}(\cdot))
\otimes v(g^{-1}(\cdot)) \otimes (gw),$$ $$\delta (u(\cdot) \otimes v(\cdot) \otimes w)
=u(\cdot) \otimes v(\cdot) \otimes \delta w - u(\delta (\cdot)) \otimes v(\cdot) \otimes w
- u(\cdot) \otimes v(\delta (\cdot)) \otimes w$$
where $u, v \in A^*$, $w \in A$, $\delta \in \mathfrak g$, $g\in G$.
Since $\delta (bc)=(\delta b)c+b(\delta c)$ for all $b,c \in A$, $\delta \in \mathfrak g$,
we have $\sum_{i} \mu_{1i}(b)  \mu_{2i}(c) (\delta \mu_{3i})
= \sum_{i} (\mu_{1i}(\delta b)  \mu_{2i}(c) \mu_{3i}+\mu_{1i}(b)  \mu_{2i}(\delta c) \mu_{3i})$.
Hence $\delta  \mu = 0$ for all $\delta  \in \mathfrak g$, and
 $\mathfrak g \mu = 0$. By~\cite[Theorem~13.2]{HumphreysAlgGr}, $G\mu = \mu$.
 Hence $g(bc)=(gb)(gc)$ and $G$ acts on $A$ by automorphisms. Using~\cite[Theorem~13.2]{HumphreysAlgGr} once again,
 we get that $G$ and $\mathfrak g$ have in $A$ the same invariant subspaces. 
\end{proof}

\subsection{Wedderburn~--- Mal'cev and Levi decompositions for algebras with derivations}

Theorem~\ref{TheoremDerAutConnection} enables us to replace the action of a semisimple Lie algebra by derivations
with an action of a semisimple affine algebraic group by automorphisms. Hence~\cite[Corollary~2.10]{SteVanOyst}
(or Theorem~\ref{TheoremAffAlgGrWedderburn}) implies 

\begin{theorem}\label{TheoremDerWedderburn}
Let $A$ be a finite dimensional associative algebra and $\mathfrak g$ be a finite dimensional semisimple Lie algebra over an algebraically closed field $F$ of characteristic $0$.
 Suppose $\mathfrak g$ is acting on $A$ by derivations.
  Then there exists a maximal semisimple subalgebra $B$ in $A$ such that
$A=B\oplus J(A)$ (direct sum of $\mathfrak g$-submodules).
\end{theorem}

Analogously, \cite[Theorem~5]{ASGordienko4} (or Theorem~\ref{TheoremAffAlgGrLevi}) implies 

\begin{theorem}\label{TheoremDerLevi}
Let $L$ and $\mathfrak g$ be finite dimensional Lie algebras over an algebraically closed field $F$ of characteristic $0$.  Suppose $\mathfrak g$ is semisimple and acting on $L$ by derivations.
  Then there exists a maximal semisimple subalgebra $B$ in $L$ such that
$L=B\oplus R$ (direct sum of $\mathfrak g$-submodules) where $R$ is the solvable radical of $L$.
\end{theorem}

\section{Polynomial $H$-identities, identities with derivations, and their codimensions}\label{SectionDerH}

We introduce polynomial identities with derivations as a particular case of polynomial
$H$-identities.

An algebra $A$
over a field $F$
is an \textit{$H$-module algebra}
or an \textit{algebra with an $H$-action},
if $A$ is endowed with a homomorphism $H \to \End_F(A)$ such that
$h(ab)=(h_{(1)}a)(h_{(2)}b)$
for all $h \in H$, $a,b \in A$. Here we use Sweedler's notation
$\Delta h = h_{(1)} \otimes h_{(2)}$ where $\Delta$ is the comultiplication
in $H$.
We refer the reader to~\cite{Danara, Montgomery, Sweedler}
   for an account
  of Hopf algebras and algebras with Hopf algebra actions.
  
    \subsection{Polynomial $H$-identities of $H$-module Lie algebras} 
  Let $F \lbrace X \rbrace$ be the absolutely free nonassociative algebra
   on the set $X := \lbrace x_1, x_2, x_3, \ldots \rbrace$.
  Then $F \lbrace X \rbrace = \bigoplus_{n=1}^\infty F \lbrace X \rbrace^{(n)}$
  where $F \lbrace X \rbrace^{(n)}$ is the linear span of all monomials of total degree $n$.
   Let $H$ be a Hopf algebra over a field $F$. Consider the algebra $$F \lbrace X | H\rbrace
   :=  \bigoplus_{n=1}^\infty H^{{}\otimes n} \otimes F \lbrace X \rbrace^{(n)}$$
   with the multiplication $(u_1 \otimes w_1)(u_2 \otimes w_2):=(u_1 \otimes u_2) \otimes w_1w_2$
   for all $u_1 \in  H^{{}\otimes j}$, $u_2 \in  H^{{}\otimes k}$,
   $w_1 \in F \lbrace X \rbrace^{(j)}$, $w_2 \in F \lbrace X \rbrace^{(k)}$.
We use the notation $$x^{h_1}_{i_1}
x^{h_2}_{i_2}\ldots x^{h_n}_{i_n} := (h_1 \otimes h_2 \otimes \ldots \otimes h_n) \otimes x_{i_1}
x_{i_2}\ldots x_{i_n}$$ (the arrangements of brackets on $x_{i_j}$ and on $x^{h_j}_{i_j}$
are the same). Here $h_1 \otimes h_2 \otimes \ldots \otimes h_n \in H^{{}\otimes n}$,
$x_{i_1} x_{i_2}\ldots x_{i_n} \in F \lbrace X \rbrace^{(n)}$. 

Note that if $(\gamma_\beta)_{\beta \in \Lambda}$ is a basis in $H$, 
then $F \lbrace X | H\rbrace$ is isomorphic to the absolutely free nonassociative algebra over $F$ with free formal  generators $x_i^{\gamma_\beta}$, $\beta \in \Lambda$, $i \in \mathbb N$.
 
    Define on $F \lbrace X | H\rbrace$ the structure of a left $H$-module
   by $$h\,(x^{h_1}_{i_1}
x^{h_2}_{i_2}\ldots x^{h_n}_{i_n})=x^{h_{(1)}h_1}_{i_1}
x^{h_{(2)}h_2}_{i_2}\ldots x^{h_{(n)}h_n}_{i_n},$$
where $h_{(1)}\otimes h_{(2)} \otimes \ldots \otimes h_{(n)}$
is the image of $h$ under the comultiplication $\Delta$
applied $(n-1)$ times, $h\in H$. Then $F \lbrace X | H\rbrace$ is \textit{the absolutely free $H$-module nonassociative algebra} on $X$, i.e. for each map $\psi \colon X \to A$ where $A$ is an $H$-module algebra,
there exists a unique homomorphism $\bar\psi \colon 
F \lbrace X | H\rbrace \to A$ of algebras and $H$-modules, such that $\bar\psi\bigl|_X=\psi$.
Here we identify $X$ with the set $\lbrace x^1_j \mid j \in \mathbb N\rbrace \subset F \lbrace X | H\rbrace$.

Consider the $H$-invariant ideal $I$ in $F\lbrace X | H \rbrace$
generated by the set \begin{equation}\label{EqSetOfHGen}
\bigl\lbrace u(vw)+v(wu)+w(uv) \mid u,v,w \in  F\lbrace X | H \rbrace\bigr\rbrace \cup\bigl\lbrace u^2 \mid u \in  F\lbrace X | H \rbrace\bigr\rbrace.
\end{equation}
 Then $L(X | H) := F\lbrace X | H \rbrace/I$
is \textit{the free $H$-module Lie algebra}
on $X$, i.e. for any $H$-module Lie algebra $L$ 
and a map $\psi \colon X \to L$, there exists a unique homomorphism $\bar\psi \colon L(X | H) \to L$
of algebras and $H$-modules such that $\bar\psi\bigl|_X =\psi$. 
 We refer to the elements of $L(X | H)$ as \textit{Lie $H$-polynomials}.


\begin{remark} If $H$ is cocommutative and $\ch F \ne 2$, then $L(X | H)$ is the ordinary
free Lie algebra with free generators $x_i^{\gamma_\beta}$, $\beta \in \Lambda$, $i \in \mathbb N$
where   $(\gamma_\beta)_{\beta \in \Lambda}$ is a basis in $H$, since the ordinary ideal of 
$F\lbrace X | H \rbrace$ generated by~(\ref{EqSetOfHGen})
is already $H$-invariant.
However, if $h_{(1)} \otimes h_{(2)} \ne h_{(2)} \otimes h_{(1)}$ for some $h \in H$,
we still have $$[x^{h_{(1)}}_i, x^{h_{(2)}}_j]=h[x_i, x_j]=-h[x_j, x_i]=-[x^{h_{(1)}}_j, x^{h_{(2)}}_i]
= [x^{h_{(2)}}_i, x^{h_{(1)}}_j]$$ in $L(X | H)$ for all $i,j \in\mathbb N$,
i.e. in the case $h_{(1)} \otimes h_{(2)} \ne h_{(2)} \otimes h_{(1)}$ the algebra $L(X | H)$ is not free as an ordinary Lie algebra.
\end{remark}

Let $L$ be an $H$-module Lie algebra for
some Hopf algebra $H$ over a field $F$.
 An $H$-polynomial
 $f \in L ( X | H )$
 is a \textit{$H$-identity} of $L$ if $\psi(f)=0$
for all homomorphisms $\psi \colon L(X|H) \to L$
of algebras and $H$-modules. In other words, $f(x_1, x_2, \ldots, x_n)$
 is a polynomial $H$-identity of $L$
if and only if $f(a_1, a_2, \ldots, a_n)=0$ for any $a_i \in L$.
 In this case we write $f \equiv 0$.
The set $\Id^H(L)$ of all polynomial $H$-identities
of $L$ is an $H$-invariant ideal of $L(X|H)$.

Denote by $V^H_n$ the space of all multilinear Lie $H$-polynomials
in $x_1, \ldots, x_n$, $n\in\mathbb N$, i.e.
$$V^{H}_n = \langle [x^{h_1}_{\sigma(1)},
x^{h_2}_{\sigma(2)}, \ldots, x^{h_n}_{\sigma(n)}]
\mid h_i \in H, \sigma\in S_n \rangle_F \subset L( X | H ).$$
Then the number $c^H_n(L):=\dim\left(\frac{V^H_n}{V^H_n \cap \Id^H(L)}\right)$
is called the $n$th \textit{codimension of polynomial $H$-identities}
or the $n$th \textit{$H$-codimension} of $L$.

\subsection{Polynomial $H$-identities of associative algebras with a generalized $H$-action}  
  In the case of associative algebras we need a more general definition.
  Let $H$ be an arbitrary associative algebra with $1$ over a field $F$.
We say that an associative algebra $A$ is an algebra with a \textit{generalized $H$-action}
if $A$ is endowed with a homomorphism $H \to \End_F(A)$
and for every $h \in H$ there exist $h'_i, h''_i, h'''_i, h''''_i \in H$
such that 
$$
h(ab)=\sum_i\bigl((h'_i a)(h''_i b) + (h'''_i b)(h''''_i a)\bigr) \text{ for all } a,b \in A.
$$

Let $F \langle X \rangle$ be the free associative algebra without $1$
   on the set $X := \lbrace x_1, x_2, x_3, \ldots \rbrace$.
  Then $F \langle X \rangle = \bigoplus_{n=1}^\infty F \langle X \rangle^{(n)}$
  where $F \langle X \rangle^{(n)}$ is the linear span of all monomials of total degree $n$.
   Let $H$ be an arbitrary associative algebra with $1$ over $F$. Consider the algebra $$F \langle X | H\rangle
   :=  \bigoplus_{n=1}^\infty H^{{}\otimes n} \otimes F \langle X \rangle^{(n)}$$
   with the multiplication $(u_1 \otimes w_1)(u_2 \otimes w_2):=(u_1 \otimes u_2) \otimes w_1w_2$
   for all $u_1 \in  H^{{}\otimes j}$, $u_2 \in  H^{{}\otimes k}$,
   $w_1 \in F \langle X \rangle^{(j)}$, $w_2 \in F \langle X \rangle^{(k)}$.
We use the notation $$x^{h_1}_{i_1}
x^{h_2}_{i_2}\ldots x^{h_n}_{i_n} := (h_1 \otimes h_2 \otimes \ldots \otimes h_n) \otimes x_{i_1}
x_{i_2}\ldots x_{i_n}.$$ Here $h_1 \otimes h_2 \otimes \ldots \otimes h_n \in H^{{}\otimes n}$,
$x_{i_1} x_{i_2}\ldots x_{i_n} \in F \langle X \rangle^{(n)}$. 

Note that if $(\gamma_\beta)_{\beta \in \Lambda}$ is a basis in $H$, 
then $F\langle X | H \rangle$ is isomorphic to the free associative algebra over $F$ with free formal  generators $x_i^{\gamma_\beta}$, $\beta \in \Lambda$, $i \in \mathbb N$.
 We refer to the elements
 of $F\langle X | H \rangle$ as \textit{associative $H$-polynomials}.
Note that here we do not consider any $H$-action on $F \langle X | H \rangle$.

Let $A$ be an associative algebra with a generalized $H$-action.
Any map $\psi \colon X \to A$ has a unique homomorphic extension $\bar\psi
\colon F \langle X | H \rangle \to A$ such that $\bar\psi(x_i^h)=h\psi(x_i)$
for all $i \in \mathbb N$ and $h \in H$.
 An $H$-polynomial
 $f \in F\langle X | H \rangle$
 is an \textit{$H$-identity} of $A$ if $\bar\psi(f)=0$
for all maps $\psi \colon X \to A$. In other words, $f(x_1, x_2, \ldots, x_n)$
 is an $H$-identity of $A$
if and only if $f(a_1, a_2, \ldots, a_n)=0$ for any $a_i \in A$.
 In this case we write $f \equiv 0$.
The set $\Id^{H}(A)$ of all $H$-identities
of $A$ is an ideal of $F\langle X | H \rangle$.

We denote by $P^H_n$ the space of all multilinear $H$-polynomials
in $x_1, \ldots, x_n$, $n\in\mathbb N$, i.e.
$$P^{H}_n = \langle x^{h_1}_{\sigma(1)}
x^{h_2}_{\sigma(2)}\ldots x^{h_n}_{\sigma(n)}
\mid h_i \in H, \sigma\in S_n \rangle_F \subset F \langle X | H \rangle.$$
Then the number $c^H_n(A):=\dim\left(\frac{P^H_n}{P^H_n \cap \Id^H(A)}\right)$
is called the $n$th \textit{codimension of polynomial $H$-identities}
or the $n$th \textit{$H$-codimension} of $A$.

\begin{remark}
One can treat polynomial $H$-identities of Lie and associative algebras as identities of nonassociative
algebras (i.e. use $F\lbrace X | H\rbrace$ instead of $F\langle X | H\rangle$ and $L(X | H)$) and define their codimensions.
However these codimensions will coincide since the $n$th $H$-codimension of $A$
equals the dimension of the subspace in $\Hom_F(A^{{}\otimes n}; A)$ that consists of those $n$-linear functions that can be represented by $H$-polynomials.
\end{remark}

\begin{theorem}\label{TheoremAssGenHopf}
Let $A$ be a finite dimensional non-nilpotent associative algebra
with a generalized $H$-action
over an algebraically closed field $F$ of characteristic $0$.  Here $H$ is an
associative algebra with $1$, not necessarily finite dimensional, acting on $A$ in such a way that the Jacobson radical $J:=J(A)$
is $H$-invariant and $A = B \oplus J$
(direct sum of $H$-submodules) where $B=B_1 \oplus \ldots \oplus B_q$ (direct sum of $H$-invariant ideals),
$B_i$ are $H$-simple semisimple algebras.
 Then there exist constants $C_1, C_2 > 0$, $r_1, r_2 \in \mathbb R$ such that $$C_1 n^{r_1} d^n \leqslant c^{H}_n(A) \leqslant C_2 n^{r_2} d^n \text{ for all }n \in \mathbb N.$$
Here
$$ d:= \max(\dim(
B_{i_1} \oplus B_{i_2} \oplus \ldots \oplus
B_{i_r}) \mid B_{i_1}J B_{i_2}J
\ldots JB_{i_r}\ne 0,$$ \begin{equation}\label{EqPIexp} 1 \leqslant i_k \leqslant q,
  1 \leqslant k \leqslant r;\
0 \leqslant r \leqslant q )
.\end{equation}
 \end{theorem}
 \begin{proof}
This theorem was proved in~\cite[Theorem~5]{ASGordienko3} under the hypothesis~$\dim H < +\infty$. We now show how to remove this restriction.
 
 Denote by $\zeta \colon H \to \End_F(A)$
the homomorphism corresponding to the $H$-action.
 Then $A$ is an algebra with a generalized $\zeta(H)$-action, and $B_i$ are $\zeta(H)$-simple.
 
 We claim that $c^H_n(A) = c^{\zeta(H)}_n(A)$ for all $n \in \mathbb N$. 
Let $\psi \colon F \langle X \mid H \rangle \to F \langle X \mid \zeta(H) \rangle$
be the homomorphism defined by $\psi(x^h)=x^{\zeta(h)}$, $h\in H$. Note that $\psi(P^H_n)=P^{\zeta(H)}_n$. Moreover $\psi(\Id^H(A))=\Id^{\zeta(H)}(A)$ since 
every $h\in H$ acts on $A$ by the operator~$\zeta(h)$.
Hence $$F\langle X \mid H \rangle/\Id^H(A) \cong F\langle X \mid \zeta(H) \rangle/\Id^{\zeta(H)}(A)$$ and $c^H_n(A) = c^{\zeta(H)}_n(A)$.

 We notice that $\dim \zeta(H) < +\infty$ and apply~\cite[Theorem~5]{ASGordienko3}
 to $\zeta(H)$-codimensions.
 \end{proof}
  
\subsection{Differential identities}
  Here we are interested in the following particular case.
 Suppose a Lie algebra $\mathfrak{g}$ is acting on a Lie or associative algebra $A$ 
 by derivations. Then $A$ is an $U(\mathfrak{g})$-module algebra
 where $U(\mathfrak{g})$ is the universal enveloping algebra of $\mathfrak{g}$,
 which is a Hopf algebra: the comultiplication $\Delta$ is defined by $\Delta(a)=1\otimes a + a \otimes 1$, the counit $\varepsilon$ is defined by $\varepsilon(a)=0$, and the antipode $S$ is defined by $Sa = -a$
 for all $a\in \mathfrak{g}$.
 The elements of $\Id^{U(\mathfrak{g})}(A)$ are called \textit{polynomial identities with derivations}
 or \textit{differential identities} of $A$  and $c^{U(\mathfrak{g})}_n(A)$ are called \textit{differential codimensions}.

\begin{example}
Consider the adjoint representation of $\mathfrak{gl}_2(F)$ on $M_2(F)$ and $\mathfrak{sl}_2(F)$.
Denote by $e_{ij}$ the matrix units.
Then $$x^{e_{11}}+x^{e_{22}} \in \Id^{U(\mathfrak{gl}_2(F))}(M_2(F)),\ \Id^{U(\mathfrak{gl}_2(F))}(\mathfrak{sl}_2(F))$$
since $a^{e_{11}}+a^{e_{22}}=[e_{11},a]+[e_{22},a]=[e_{11}+e_{22},a]=0$
for all $a \in M_2(F)$.
\end{example}

The analog of Amitsur's conjecture for codimensions of polynomial identities with derivations can be formulated
as follows.

\begin{conjecture} Let $A$ be a Lie or associative algebra with an action of a Lie algebra $\mathfrak g$
by derivations. Then there exists
 $\PIexp^{U(\mathfrak{g})}(A):=\lim\limits_{n\to\infty}
 \sqrt[n]{c^{U(\mathfrak{g})}_n(A)} \in \mathbb Z_+$.
\end{conjecture}

\begin{remark}
I.B.~Volichenko~\cite{Volichenko} gave an example
of an infinite dimensional Lie algebra~$L$ with
a nontrivial polynomial identity for which the
 growth of codimensions~$c_n(L)$ of ordinary polynomial identities
 is overexponential.
 M.V.~Zaicev and S.P.~Mishchenko~\cite{VerZaiMishch, ZaiMishch}
  gave an example
of an infinite dimensional Lie algebra~$L$
 with a nontrivial polynomial identity
 such that
there exists fractional
  $\PIexp(L):=\lim\limits_{n\to\infty} \sqrt[n]{c_n(L)}$.
\end{remark}

We claim that the following theorems hold:

\begin{theorem}\label{TheoremDer}
Let $A$ be a finite dimensional non-nilpotent Lie or associative algebra
over an field $F$ of characteristic $0$. Suppose a finite dimensional
semisimple Lie algebra $\mathfrak g$ acts on $A$ by derivations.
 Then there exist constants $C_1, C_2 > 0$, $r_1, r_2 \in \mathbb R$, $d \in \mathbb N$ such that $C_1 n^{r_1} d^n \leqslant c^{U(\mathfrak{g})}_n(A) \leqslant C_2 n^{r_2} d^n$ for all $n \in \mathbb N$.
\end{theorem}

\begin{remark}
If $A$ is nilpotent, i.e. $x_1 \ldots x_p\equiv 0$ for some $p\in\mathbb N$, then
$P^{U(\mathfrak{g})}_n \subseteq \Id^{U(\mathfrak{g})}(A)$ and $c^{U(\mathfrak{g})}_n(A)=0$ for all $n \geqslant p$.
\end{remark}

\begin{corollary}
The above analog of Amitsur's conjecture holds
 for such codimensions.
\end{corollary}

\begin{theorem}\label{TheoremDerSum}
Let $A=A_1 \oplus \ldots \oplus A_s$ (direct sum of ideals) be a finite dimensional Lie or associative
 algebra over a field $F$ of characteristic $0$.
Suppose a finite dimensional
semisimple Lie algebra $\mathfrak g$ acts on $A$ by derivations in such a way that $A_i$ are invariant. 
  Then $\PIexp^{U(\mathfrak{g})}(A)=\max_{1 \leqslant i \leqslant s}
\PIexp^{U(\mathfrak{g})}(A_i)$.
\end{theorem}

Theorems~\ref{TheoremDer} and~\ref{TheoremDerSum} will be proved in Subsection~\ref{SubsectionGApplToDiff}.

\section{Polynomial $G$-identities and their codimensions}\label{SectionG}

\subsection{Definitions and theorems}
 
Let $G$ be a group
with a fixed (normal) subgroup $G_0$ of index ${}\leqslant 2$.
 Denote by $F\langle X | G \rangle$
the free associative algebra over $F$ with free formal generators $x^g_j$, $j\in\mathbb N$,
 $g \in G$. Here $X := \lbrace x_1, x_2, x_3, \ldots \rbrace$, $x_j := x_j^1$. Define
 $$(x_{i_1}^{g_1} x_{i_2}^{g_2}\ldots x_{i_{n-1}}^{g_{n-1}}
  x_{i_n}^{g_n})^h :=
x_{i_1}^{hg_1} x_{i_2}^{hg_2}\ldots x_{i_{n-1}}^{hg_{n-1}} x_{i_n}^{hg_n}
\text { for } h \in G_0, $$
 $$(x_{i_1}^{g_1} x_{i_2}^{g_2}\ldots x_{i_{n-1}}^{g_{n-1}}
  x_{i_n}^{g_n})^h :=
x_{i_n}^{hg_n} x_{i_{n-1}}^{hg_{n-1}}  \ldots x_{i_2}^{hg_2}x_{i_1}^{hg_1}
\text { for } h \in G\backslash G_0.$$
  Then $F\langle X | G \rangle$ becomes the free $G$-algebra with
 free generators $x_j$, $j \in \mathbb N$. We call its elements
\textit{$G$-polynomials}.
 Let $A$ be an associative  $G$-algebra over $F$
 such that $G_0 \subseteq G$ is acting on $A$ by automorphisms
 and the elements of $G\backslash G_0$ are acting on $A$ by anti-automorphisms.
  A $G$-polynomial
 $f(x_1, \ldots, x_n)\in F\langle X | G \rangle$
 is a \textit{$G$-identity} of $A$ if $f(a_1, \ldots, a_n)=0$
for all $a_i \in A$. In this case we write
$f \equiv 0$.
The set $\Id^{G}(A)$ of all $G$-identities
of $A$ is an ideal in $F\langle X | G \rangle$ invariant under $G$-action.

\begin{example}\label{ExampleIdG} Let $M_2(F)$ be the algebra
of $2\times 2$ matrices. Consider $\psi \in \Aut(M_2(F))$
defined by the formula $$\left(
\begin{array}{cc}
a & b \\
c & d
\end{array}
 \right)^\psi := \left(
\begin{array}{rr}
a & -b \\
-c & d
\end{array}
 \right).$$
Then $[x+x^{\psi},y+y^{\psi}]\in \Id^{G}(M_2(F))$
where
$G=\langle \psi \rangle \cong \mathbb Z_2$.
Here $[x,y]:=xy-yx$.
\end{example}

Denote by $P^G_n$ the space of all multilinear $G$-polynomials
in $x_1, \ldots, x_n$, $n\in\mathbb N$, i.e.
$$P^{G}_n = \langle x^{g_1}_{\sigma(1)}
x^{g_2}_{\sigma(2)}\ldots x^{g_n}_{\sigma(n)}
\mid g_i \in G, \sigma\in S_n \rangle_F \subset F \langle X | G \rangle$$
where $S_n$ is the $n$th symmetric group.
Then the number $c^G_n(A):=\dim\left(\frac{P^G_n}{P^G_n \cap \Id^G(A)}\right)$
is called the $n$th \textit{codimension of polynomial $G$-identities}
or the $n$th \textit{$G$-codimension} of $A$.

If $L$ is a Lie algebra with $G$-action, we define polynomial $G$-identities
and their codimensions analogously, replacing in our definition the free
associative algebra by the free Lie one.

If $G$ is trivial, we get ordinary polynomial identities and their codimensions. Note also that if $A$ is a $G$-algebra, then $A$ is an algebra with a generalized $FG$-action and $c^{FG}_n(A) = c^G_n(A)$ for all $n \in \mathbb N$.

The analog of Amitsur's conjecture for $G$-codimensions can be formulated
as follows.

\begin{conjecture} There exists
 $\PIexp^G(A):=\lim\limits_{n\to\infty}
 \sqrt[n]{c^G_n(A)} \in \mathbb Z_+$.
\end{conjecture}

In the Lie case, we have the following two results:

\begin{theorem}[{\cite[Theorem 3]{ASGordienko5}}]\label{TheoremLieGAffAlg}
Let $L$ be a finite dimensional non-nilpotent Lie algebra
over an algebraically closed field $F$ of characteristic $0$. Suppose a reductive affine algebraic group
  $G$ acts on $L$ rationally by automorphisms and anti-automorphisms. Then there exist
 constants $C_1, C_2 > 0$, $r_1, r_2 \in \mathbb R$,
  $d \in \mathbb N$ such that
   $C_1 n^{r_1} d^n \leqslant c^{G}_n(L)
    \leqslant C_2 n^{r_2} d^n$ for all $n \in \mathbb N$.
\end{theorem}

\begin{theorem}[{\cite[Theorem 5]{ASGordienko5}}]\label{TheoremLieGAffAlgSum}
Let $L=L_1 \oplus \ldots \oplus L_s$ (direct sum of ideals) be a finite dimensional Lie algebra
over an algebraically closed field $F$ of characteristic $0$. Suppose a reductive affine algebraic group
  $G$ acts on $L$ rationally by automorphisms and anti-automorphisms and the ideals 
$L_i$ are $G$-invariant. Then $\PIexp^G(L)=\max_{1 \leqslant i \leqslant s}
\PIexp^G(L_i)$.
\end{theorem}

In particular, for reductive $G$, the analog of Amitsur's conjecture holds for $G$-codimensions of finite dimensional Lie algebras. In this subsection we will derive similar results for associative algebras:

\begin{theorem}\label{TheoremAssAffAlgG}
Let $A$ be a finite dimensional non-nilpotent associative algebra
over an algebraically closed field $F$ of characteristic $0$. Suppose a reductive affine algebraic group
  $G$ acts on $A$ rationally by automorphisms and anti-automorphisms.
 Then there exist constants $C_1, C_2 > 0$, $r_1, r_2 \in \mathbb R$, $d \in \mathbb N$ such that $C_1 n^{r_1} d^n \leqslant c^{G}_n(A) \leqslant C_2 n^{r_2} d^n$ for all $n \in \mathbb N$.
\end{theorem}

\begin{corollary}
The above analog of Amitsur's conjecture holds
 for such codimensions.
\end{corollary}


\begin{theorem}\label{TheoremAssGAffAlgSum}
Let $A=A_1 \oplus \ldots \oplus A_s$ (direct sum of ideals) be a finite dimensional associative algebra
over an algebraically closed field $F$ of characteristic $0$. Suppose a reductive affine algebraic group
  $G$ acts on $A$ rationally by automorphisms and anti-automorphisms, and the ideals 
$A_i$ are $G$-invariant. Then $\PIexp^G(A)=\max_{1 \leqslant i \leqslant s}
\PIexp^G(A_i)$.
\end{theorem}

We need the following result, which is similar to \cite[Theorem~6]{ASGordienko4} in the Lie case.

\begin{lemma}\label{LemmaHWedderburn}
Let $A$ be a finite dimensional semisimple associative $H$-module algebra
where $H$ is a Hopf algebra over an arbitrary field $F$ such that the antipode $S$
is bijective. Then $A=B_1 \oplus \ldots \oplus B_q$ (direct sum of $H$-invariant ideals)
where $B_i$ are $H$-simple algebras.
\end{lemma}
\begin{proof}
By Wedderburn's theorem, $A=A_1 \oplus \ldots \oplus A_s$ (direct sum of ideals) 
where $A_i$ are simple algebras not necessarily $H$-invariant.
Let $B_1$ be a minimal $H$-invariant ideal of $A$. Then $B_1 = A_{i_1} \oplus \ldots \oplus A_{i_k}$
for some $i_1, i_2, \ldots, i_k \in \lbrace 1, 2, \ldots, s\rbrace$. 
Consider $\tilde B_1 = \lbrace a \in A \mid ab=ba=0 \text{ for all } b \in B_1 \rbrace$.
Then $\tilde B_1$ equals the sum of all $A_j$, $j \notin \lbrace 
i_1, i_2, \ldots, i_k\rbrace$, and $A = B_1 \oplus \tilde B_1$.
 We claim that $\tilde B_1$ is $H$-invariant. Indeed, let $a \in \tilde B_1$,
 $b \in B_1$. Denote by $\varepsilon$ the counit of $H$
 and by $\Delta(h)=h_{(1)}\otimes h_{(2)}$ its comultiplication.
   Then 
  $$(ha)b = (h_{(1)}a)(\varepsilon(h_{(2)}) b)=(h_{(1)}a)(h_{(2)}(Sh_{(3)})b)
  = h_{(1)}(a(Sh_{(2)})b)=0$$
  since $B_1$ is $H$-invariant.
  Moreover, $$b(ha) = (S^{-1}(\varepsilon(h_{(1)})1)b)(h_{(2)}a)=(S^{-1}(h_{(1)}Sh_{(2)})b)(h_{(3)}a)=$$
  $$(h_{(2)}(S^{-1}h_{(1)})b)(h_{(3)}a)
  = h_{(2)}(((S^{-1}h_{(1)})b)a)=0.$$
  Hence $\tilde B_1$ is $H$-invariant and the inductive argument finishes the proof.
\end{proof}
\begin{lemma}\label{LemmaGWedderburn}
Let $A$ be a finite dimensional semisimple associative algebra
over an arbitrary field $F$, with an action of a group $G$ by automorphisms and anti-automorphisms. Then $A=B_1 \oplus \ldots \oplus B_q$ (direct sum of $G$-invariant ideals)
where $B_i$ are $G$-simple algebras.
\end{lemma}
\begin{proof}
Again, suppose that $G_0 \subseteq G$ is acting on $A$ by automorphisms and the elements of $G\backslash G_0$
are acting by anti-automorphisms. 

If $G = G_0$, the lemma is a consequence of Lemma~\ref{LemmaHWedderburn}, since the antipode $S$ of the Hopf algebra $FG$ is bijective: $Sg=g^{-1}$, $g\in G$.

Suppose $G \ne G_0$. Then by Lemma~\ref{LemmaHWedderburn},
$B = \tilde B_1 \oplus \ldots \oplus \tilde B_k$ (direct sum of $G_0$-invariant ideals) where $\tilde B_i$
are $G_0$-simple algebras. Standard arguments (see e.g.~\cite[Chapter~III, Section~5, Theorem~4]{JacobsonLie}) show that every $G_0$-simple ideal of $B$ coincides
with one of $\tilde B_i$. Let $g \in G\backslash G_0$. Then $(\tilde B_i + g \tilde B_i)$ is a $G$-simple ideal
for every $1 \leqslant i \leqslant k$ and $A=B_1 \oplus \ldots \oplus B_q$ (direct sum of $G$-invariant ideals)
where each $B_j=\tilde B_i + g \tilde B_i$ for some $i$.
\end{proof}
\begin{proof}[Proof of Theorems~\ref{TheoremAssAffAlgG} and~\ref{TheoremAssGAffAlgSum}]
Note that, by Theorem~\ref{TheoremAffAlgGrWedderburn}, $A=B\oplus J(A)$ (direct sum of $G$-invariant subspaces)
where $B$ is a $G$-invariant maximal semisimple subalgebra.  Hence Lemma~\ref{LemmaGWedderburn} implies $B=B_1 \oplus \ldots \oplus B_q$ (direct sum of $G$-invariant spaces) where $B_i$ are $G$-simple algebras.  Now Theorem~\ref{TheoremAssAffAlgG} follows  from Theorem~\ref{TheoremAssGenHopf}.

Theorem~\ref{TheoremAssGAffAlgSum} is an immediate consequence of~(\ref{EqPIexp}).
\end{proof}

\subsection{Applications to differential identities}\label{SubsectionGApplToDiff}
Let $C$ be a vector space and let $C^*$ be its dual. We say that a subspace $A \subseteq C^*$ is \textit{dense} in $C^*$ if $A^\perp=0$ where $A^\perp := \lbrace c \in C \mid \varphi(c)=0 \text{ for all } \varphi \in A\rbrace$. An equivalent condition for $A$ is to separate points of $C$.

\begin{lemma}\label{LemmaKochetov}
Let $V$ be a finite dimensional right comodule over a coalgebra $C$ over a field $F$. Denote by $\zeta \colon C^* \to \End_F(V)$
the homomorphism corresponding to the left $C^*$-module structure on $V$ where $C^*$ is the algebra dual to $C$.
Suppose $A$ is a dense subalgebra of $C^*$. Then $\zeta(A)=\zeta(C^*)$.
\end{lemma}
\begin{proof}
Let $(v_i)_{1 \leqslant i \leqslant \dim V}$ be a basis of $V$. Denote by $\rho \colon V \to V \otimes C$
the comodule map of $V$. Let $\rho(v_i)=\sum_{j=1}^{\dim V} v_j \otimes c_{ji}$ where $c_{ij} \in C$, $1 \leqslant i,j \leqslant \dim V$. Denote $$D = \langle c_{ij} \mid 1 \leqslant i,j \leqslant \dim V \rangle_F
.$$ 
 Let $\pi \colon C^* \to D^*$ be the natural projection. 
We claim that $\pi(A)=D^*$. Indeed, if $\pi(A) \ne D^*$, then there exists $c \in D$, $c\ne 0$,
 $\varphi(c) = 0$ for all $\varphi \in A$. We get a contradiction with $A^\perp = 0$.
 Suppose $\varphi \in C^*$. Choose $\tilde \varphi \in A$ such that $\pi(\varphi)=\pi(\tilde\varphi)$.
 Then
 $$\zeta(\varphi)v=\varphi(v_{(1)})v_{(0)}=\pi(\varphi)(v_{(1)})v_{(0)}=
 \pi(\tilde\varphi)(v_{(1)})v_{(0)}=\zeta(\tilde\varphi)v$$
 for every $v\in V$.
 Hence $\zeta(A)=\zeta(C^*)$.
\end{proof}

Note that~\cite[Propositions~9.2.10, 9.2.5, and Example 9.2.8]{Montgomery} imply

\begin{lemma}\label{LemmaUgDense}
Let $G$ be a connected affine algebraic group over an algebraically closed field $F$ 
of characteristic $0$ and let $\mathfrak g$ be its Lie algebra.
Then $U(\mathfrak g)$ is dense in $\mathcal O(G)^*$.
\end{lemma}

Using Lemma~\ref{LemmaUgDense}, we get

\begin{lemma}\label{LemmaCodimGUg}
Let $G$ be a connected affine algebraic group over an algebraically closed field $F$ 
of characteristic $0$ and let $\mathfrak g$ be its Lie algebra.
Suppose $G$ is acting by automorphisms on a finite dimensional algebra $A$.
Then $c_n^{U(\mathfrak g)}(A) = c_n^G(A)$ for all $n \in \mathbb N$. 
\end{lemma}
\begin{proof}
First, we notice that $A$ is an $\mathcal O(G)$-comodule algebra. The actions of the algebras $FG$
and $U(\mathfrak g)$ on $A$ can be induced from the $\mathcal O(G)$-comodule structure
and the natural maps from $FG$ and $U(\mathfrak g)$ to $\mathcal O(G)^*$.
Obviously, the image of $FG$ is dense in $\mathcal O(G)^*$.
By Lemma~\ref{LemmaUgDense}, the image of $U(\mathfrak g)$ is dense in $\mathcal O(G)^*$.
Hence, by Lemma~\ref{LemmaKochetov}, $FG$ and $U(\mathfrak g)$ are acting on $A$ by the same operators.
To be specific in notation, assume that $A$ is associative. (If $A$ is a Lie algebra, we use the spaces
$V_n$ instead of $P_n$.)
We can treat $\frac{P^G_n}{P^G_n \cap \Id^G(A)}$ and 
$\frac{P^{U(\mathfrak g)}_n}{P^{U(\mathfrak g)}_n \cap\ \Id^{U(\mathfrak g)}(A)}$, $n\in\mathbb N$, as 
the spaces of $n$-linear functions on $A$ that can be presented, respectively, by $G$- and $U(\mathfrak g)$-polynomials. Since the functions are the same, we get
$$c_n^G(A)=\dim\frac{P^G_n}{P^G_n \cap \Id^G(A)}=\dim\frac{P^{U(\mathfrak g)}_n}{P^{U(\mathfrak g)}_n \cap \ \Id^{U(\mathfrak g)}(A)} = c^{U(\mathfrak g)}_n(A).$$
\end{proof}
\begin{proof}[Proof of Theorems~\ref{TheoremDer} and~\ref{TheoremDerSum}.]
$H$-codimensions do not change upon an extension of the base field.
The proof is analogous to the cases of ordinary codimensions of
associative~\cite[Theorem~4.1.9]{ZaiGia} and
Lie algebras~\cite[Section~2]{ZaiLie}.
Thus without loss of generality we may assume
 $F$ to be algebraically closed.
 
Using Theorem~\ref{TheoremDerAutConnection}, we replace the $\mathfrak g$-action by $G$-action
where $G$ is a simply connected semisimple affine algebraic group. By Lemma~\ref{LemmaCodimGUg}, $c^{U(\mathfrak g)}_n(A)=c^G_n(A)$ for all $n\in\mathbb N$,
and Theorems~\ref{TheoremDer} and~\ref{TheoremDerSum}
are consequences of Theorems~\ref{TheoremLieGAffAlg}, \ref{TheoremLieGAffAlgSum}, \ref{TheoremAssAffAlgG}, and \ref{TheoremAssGAffAlgSum}.
\end{proof}

\section{Graded polynomial identities and their codimensions}\label{SectionGraded}

Let $G$ be a group and $F$ be a field. Denote by $F\langle X^{\mathrm{gr}} \rangle $ the free $G$-graded associative  algebra over $F$ on the countable set $$X^{\mathrm{gr}}:=\bigcup_{g \in G}X^{(g)},$$ $X^{(g)} = \{ x^{(g)}_1,
x^{(g)}_2, \ldots \}$,  i.e. the algebra of polynomials
 in non-commuting variables from $X^{\mathrm{gr}}$.
  The indeterminates from $X^{(g)}$ are said to be homogeneous of degree
$g$. The $G$-degree of a monomial $x^{(g_1)}_{i_1} \dots x^{(g_t)}_{i_t} \in F\langle
 X^{\mathrm{gr}} \rangle $ is defined to
be $g_1 g_2 \dots g_t$, as opposed to its total degree, which is defined to be $t$. Denote by
$F\langle
 X^{\mathrm{gr}} \rangle^{(g)}$ the subspace of the algebra $F\langle
 X^{\mathrm{gr}} \rangle$ spanned
 by all the monomials having
$G$-degree $g$. Notice that $$F\langle
 X^{\mathrm{gr}} \rangle^{(g)} F\langle
 X^{\mathrm{gr}} \rangle^{(h)} \subseteq F\langle
 X^{\mathrm{gr}} \rangle^{(gh)},$$ for every $g, h \in G$. It follows that
$$F\langle
 X^{\mathrm{gr}} \rangle =\bigoplus_{g\in G} F\langle
 X^{\mathrm{gr}} \rangle^{(g)}$$ is a $G$-grading.
  Let $f=f(x^{(g_1)}_{i_1}, \dots, x^{(g_t)}_{i_t}) \in F\langle
 X^{\mathrm{gr}} \rangle$.
We say that $f$ is
a \textit{graded polynomial identity} of
 a $G$-graded algebra $A=\bigoplus_{g\in G}
A^{(g)}$
and write $f\equiv 0$
if $f(a^{(g_1)}_{i_1}, \dots, a^{(g_t)}_{i_t})=0$
for all $a^{(g_j)}_{i_j} \in A^{(g_j)}$, $1 \leqslant j \leqslant t$.
  The set $\Id^{\mathrm{gr}}(A)$ of graded polynomial identities of
   $A$ is
a graded ideal of $F\langle
 X^{\mathrm{gr}} \rangle$.

\begin{example}\label{ExampleIdGr}
 Let $G=\mathbb Z_2 = \lbrace \bar 0, \bar 1 \rbrace$,
$M_2(F)=M_2(F)^{(\bar 0)}\oplus M_2(F)^{(\bar 1)}$
where $M_2(F)^{(\bar 0)}=\left(
\begin{array}{cc}
F & 0 \\
0 & F
\end{array}
 \right)$ and $M_2(F)^{(\bar 1)}=\left(
\begin{array}{cc}
0 & F \\
F & 0
\end{array}
 \right)$. Then  $x^{(\bar 0)} y^{(\bar 0)} - y^{(\bar 0)} x^{(\bar 0)}
\in \Id^{\mathrm{gr}}(M_2(F))$.
\end{example}

Let
$P^{\mathrm{gr}}_n := \langle x^{(g_1)}_{\sigma(1)}
x^{(g_2)}_{\sigma(2)}\ldots x^{(g_n)}_{\sigma(n)}
\mid g_i \in G, \sigma\in S_n \rangle_F \subset F \langle X^{\mathrm{gr}} \rangle$, $n \in \mathbb N$.
Then the number $$c^{\mathrm{gr}}_n(A):=\dim\left(\frac{P^{\mathrm{gr}}_n}{P^{\mathrm{gr}}_n \cap \Id^{\mathrm{gr}}(A)}\right)$$
is called the $n$th \textit{codimension of graded polynomial identities}
or the $n$th \textit{graded codimension} of $A$.

\begin{remark} Let $\tilde G \supseteq G$ be another group.
Denote by $F\langle X^{\widetilde{\mathrm{gr}}} \rangle$, $\Id^{\widetilde{\mathrm{gr}}}(A)$, $P^{\widetilde{\mathrm{gr}}}_n$, 
$c^{\widetilde{\mathrm{gr}}}_n(A)$ the objects corresponding to the $\tilde G$-grading. Let $I$ be the ideal of $F\langle X^{\widetilde{\mathrm{gr}}} \rangle$
generated by $x^{(g)}_j$, $j\in\mathbb N$, $g \notin G$.
We can identify $F\langle X^{\mathrm{gr}} \rangle$ with the corresponding subalgebra in 
$F\langle X^{\widetilde{\mathrm{gr}}}  \rangle$. 
Then $$F\langle X^{\widetilde{\mathrm{gr}}} \rangle= F\langle X^{\mathrm{gr}}  \rangle \oplus I, \qquad 
\Id^{\widetilde{\mathrm{gr}}}(A) = \Id^{\mathrm{gr}}(A) \oplus I,\qquad P^{\widetilde{\mathrm{gr}}}_n=
P^{\mathrm{gr}}_n \oplus (P^{\widetilde{\mathrm{gr}}}_n \cap I),$$ $$
P^{\widetilde{\mathrm{gr}}}_n \cap  \Id^{\widetilde{\mathrm{gr}}}(A) =
(P^{\mathrm{gr}}_n \cap \Id^{\mathrm{gr}}(A)) \oplus (P^{\widetilde{\mathrm{gr}}}_n \cap I)\qquad
\text{(direct sums of subspaces).}$$ Hence $c^{\widetilde{\mathrm{gr}}}_n(A) = c^{\mathrm{gr}}_n(A)$
for all $n\in\mathbb N$. In particular, we can always replace the grading group with 
the subgroup generated by the elements corresponding to the nonzero components.
\end{remark}

The analog of Amitsur's conjecture for graded codimensions can be formulated
as follows.

\begin{conjecture} There exists
 $\PIexp^{\mathrm{gr}}(A):=\lim\limits_{n\to\infty} \sqrt[n]{c^\mathrm{gr}_n(A)} \in \mathbb Z_+$.
\end{conjecture}

In 2011, E.~Aljadeff and A.~Giambruno~\cite{AljaGia} proved
 the analog Amitsur's conjecture for graded codimensions of all associative (not necessarily finite dimensional) PI-algebras provided that $G$ is finite.
 (When the algebra is finite dimensional, this result can be easily derived from 
 the corresponding result on $H$-codimensions, see~\cite[Sections 1.3--1.4]{ASGordienko3}.)
However, for finite dimensional $A$ and Abelian $G$,
we do not need $G$ to be finite.

\begin{theorem}\label{TheoremAssGr}
Let $A$ be a finite dimensional non-nilpotent associative algebra
over a field $F$ of characteristic $0$, graded by an Abelian group $G$ not necessarily finite.
 Then
there exist constants $C_1, C_2 > 0$, $r_1, r_2 \in \mathbb R$, $d \in \mathbb N$
such that $C_1 n^{r_1} d^n \leqslant c^{\mathrm{gr}}_n(A) \leqslant C_2 n^{r_2} d^n$
for all $n \in \mathbb N$.
\end{theorem}
\begin{corollary}
The above analog of Amitsur's conjecture holds for such codimensions.
\end{corollary}

\begin{theorem}\label{TheoremAssGrSum}
Let $A=A_1 \oplus \ldots \oplus A_s$ (direct sum of graded ideals) be a finite dimensional associative
 algebra over a field $F$ of characteristic $0$ graded by an Abelian group $G$. Then $\PIexp^{\mathrm{gr}}(A)=\max_{1 \leqslant i \leqslant s}
\PIexp^{\mathrm{gr}}(A_i)$.
\end{theorem}

%

To prove these theorems, we need the following well known facts.
Let $G$ be an Abelian group. Denote by $\hat G = \Hom(G, F^\times)$ the group of homomorphisms from $G$ into the multiplicative group $F^\times$ of the field $F$. Then each $G$-graded space $V=\bigoplus_{g\in G} V^{(g)}$ becomes an $F\hat G$-module:
$\chi v^{(g)} = \chi(g)v^{(g)}$ for all $\chi \in \hat G$ and $v^{(g)} \in V^{(g)}$.
Moreover, if $G$ is finitely generated, $F$ is algebraically closed of characteristic $0$, and
$V$ is finite dimensional, then every $\hat G$-invariant subspace in $V$ is $G$-graded.

The following  lemma is completely analogous to~\cite[Lemma~24]{ASGordienko5}:

\begin{lemma}\label{LemmaGrToHatG}
Let $A$ be a finite dimensional associative algebra
over an algebraically closed field $F$ of characteristic $0$, graded by a finitely generated Abelian group $G$.
Consider the $\hat G$-action on $A$ defined above. Then $c^{\mathrm{gr}}_n(A)=c^{\hat G}_n(A)$ for all $n\in \mathbb N$.
\end{lemma}

\begin{proof}[Proof of Theorems~\ref{TheoremAssGr} and~\ref{TheoremAssGrSum}]
Graded codimensions do not change upon an extension of the base field.
The proof is analogous to the case of ordinary codimensions~\cite[Theorem~4.1.9]{ZaiGia}.
Thus without loss of generality we may assume
 $F$ to be algebraically closed.
 
 Since the set $\lbrace g \in G \mid A^{(g)}\ne 0 \rbrace$ is finite, we may assume the grading group $G$ to be finitely generated. Hence $A$ is an $FG$-comodule algebra where $FG$ is a finitely generated commutative algebra that does not contain nonzero nilpotent elements.
 Therefore, $A$ is a rational representation of the reductive affine algebraic group $\hat G = \Hom(FG, F)$
 where $\Hom(FG, F)$ is the set of unitary algebra homomorphisms $FG \to F$.
 By Lemma~\ref{LemmaGrToHatG}, $c_n^\mathrm{gr}(A)=c_n^{\hat G}(A)$.
 Now we use Theorems~\ref{TheoremAssAffAlgG} and~\ref{TheoremAssGAffAlgSum}.
\end{proof}

\section{Formulas for PI-exponents}\label{SectionPIExp}

\subsection{Associative algebras}
Let $F$ be an algebraically
 closed field of characteristic $0$.
 Replacing in~(\ref{EqPIexp}) the $H$-invariant subalgebra and $H$-simple ideals by, respectively, a graded subalgebra
 and graded simple ideals, we obtain the formula for~$\PIexp^{\mathrm{gr}}(A)$ under the assumptions
 of Theorem~\ref{TheoremAssGr}. Replacing the $H$-invariant subalgebra and $H$-simple ideals by, respectively, a $G$-invariant subalgebra
 and $G$-simple ideals, we obtain the formula for $\PIexp^G(A)$ under the assumptions
 of Theorem~\ref{TheoremAssAffAlgG}. Analogously,
 we obtain the formula for $\PIexp^{U(\mathfrak g)}(A)$ under the assumptions
 of Theorem~\ref{TheoremDer}. We claim that $\PIexp^{U(\mathfrak g)}(A) = \PIexp(A)$
 and $\PIexp^G(A) = \PIexp(A)$ if $G$ is a connected group.
 
 \begin{lemma}\label{LemmaDerSimpleSum}
 Let $B=B_1 \oplus \ldots \oplus B_q$ (direct sum of ideals) be an algebra not necessarily associative
 over a field $F$ where $B_i$ are simple algebras. Suppose $\delta$ is a derivation of $B$. Then all $B_i$ are invariant under $\delta$.
 \end{lemma}
 \begin{proof}
 Let $1 \leqslant i \leqslant q$ and $a \in B_i$. Then $\delta(a)=\sum_{i=1}^q b_i$
 where $b_j \in B_j$, $1 \leqslant j \leqslant q$.
 For all $b \in B_j$, $j\ne i$, we have $$0=\delta(ab)=\delta(a)b+a\delta(b)=b_jb+a\delta(b).$$
 Hence $b_jb = - a\delta(b) \in B_i$ and $b_j b = 0$. Analogously, $b b_j = 0$ for all $b\in B_j$.
 Since $B_j$ is simple, we get $b_j = 0$ for all $j\ne i$ and $\delta(a)\in B_i$.
  \end{proof}
  \begin{lemma}\label{LemmaDerTrivIdeals}
  Let $A$ be a finite dimensional associative or Lie algebra over a field $F$ of characteristic $0$
  and let $\mathfrak g$ be a Lie algebra acting on $A$ by derivations. Suppose $A$ and $\lbrace 0 \rbrace$ are the only $\mathfrak g$-invariant ideals in $A$. Then either $A$ is semisimple or $A^2=0$.
  \end{lemma}
  \begin{proof} Suppose $A$ is associative. By~\cite[Lemma 3.2.2]{Dixmier}, the Jacobson radical (which coincides with the prime radical) of a finite dimensional associative algebra is invariant under all derivations. Hence either $J(A)=0$ and the lemma is proved or $A=J(A)$ is a nilpotent algebra. In the last case
  $A^2 \ne A$ is a $\mathfrak g$-invariant ideal. Hence $A^2=0$.
  
  Suppose $A$ is a Lie algebra.
  Recall that by~\cite[Chapter~III, Section~6, Theorem~7]{JacobsonLie} 
 the solvable radical $R$ of $A$ is invariant under all derivations. Hence either $R=0$ and $A$ is semisimple
 or $A$ is solvable. In the last case $[A,A]\ne A$ is invariant under all derivations. Hence $[A,A]=0$.
  \end{proof}
  \begin{lemma}\label{LemmaDerSimple}
  If $B$ is a $\mathfrak g$-simple finite dimensional associative or Lie algebra over a field $F$ of characteristic $0$ where $\mathfrak g$ is a Lie algebra acting on $B$ by derivations, then $B$ is a simple algebra.
  \end{lemma}
  \begin{proof}
  By Lemma~\ref{LemmaDerTrivIdeals}, $B$ is semisimple and $B=B_1 \oplus \ldots \oplus B_q$ (direct sum of ideals) for some simple algebras $B_i$. By Lemma~\ref{LemmaDerSimpleSum}, each $B_i$ is $\mathfrak g$-invariant. Hence $q=1$ and $B=B_1$.
   \end{proof}
  \begin{lemma}\label{LemmaGSimple}
  If $B$ is a $G$-simple finite dimensional associative or Lie algebra  over a field $F$ of characteristic $0$ where $G$ is a connected affine algebraic group rationally acting on $B$ by automorphisms and anti-automorphisms, then $B$ is a simple algebra and $G$ is acting by automorphisms only.
  \end{lemma}
  \begin{proof} Since the radicals are invariant under all automorphisms and anti-automorphisms, $B$ is semisimple and $B=B_1 \oplus \ldots \oplus B_q$  (direct sum of ideals) for some simple algebras $B_i$.
  By~\cite[Chapter~III, Section~5, Theorem~4]{JacobsonLie},
  $B_i$ are the only simple ideals of $B$. Hence there exists a homomorphism
  $\varphi \colon G \to S_n$ such that $B_i^g = B_{\varphi(g)(i)}$
  for all $1 \leqslant i \leqslant q$ and $g\in G$ where $S_n$ is the $n$th symmetric group.
  Thus $G$ is the disjoint union of closed sets corresponding to different $\varphi(g) \in S_n$.
  Since $G$ is connected, we get $B_i^g = B_i$ for all $1 \leqslant i \leqslant q$ and $g\in G$.
  Hence $q=1$, $B=B_1$ and $B$ is simple.
  
  If $\Aut(B)\ne \Aut^*(B)$, then we have a homomorphism $G \to \Aut^*(B)/\Aut(B) \cong \mathbb Z_2$,
  which is again trivial. Hence $G$ is acting by automorphisms only.
  \end{proof}
  
  \begin{theorem}\label{TheoremAssDerPIexpEqual}
  Let $A$ be a finite dimensional associative algebra
over a field $F$ of characteristic $0$. Suppose a finite dimensional
semisimple Lie algebra $\mathfrak g$ acts on $A$ by derivations.
Then $\PIexp^{U(\mathfrak g)}(A)=\PIexp(A)$.
  \end{theorem}
  \begin{proof} Again, without loss of generality we may assume $F$ to be algebraically closed.
  Now we compare~(\ref{EqPIexp}) with the formula for the ordinary PI-exponent~\cite[Section 6.2]{ZaiGia} and apply Lemma~\ref{LemmaDerSimple}.
  \end{proof}
  
  \begin{remark}
  This fact is not surprising, since if all derivations
  are inner, differential identities are a particular case of generalized polynomial identities,
  for which the exponent of the codimension growth is equal to the PI-exponent too~\cite{ASGordienkoGPI}.
  \end{remark}
  
  \begin{remark} 
We have $\PIexp^{U(\mathfrak g)}(A)=\PIexp(A)$, however the codimensions themselves can be different.  
  Suppose $\mathfrak{sl}_2(F)$ is acting on $M_2(F)$ by the adjoint representation. Then $c_1(M_2(F))=1$,
  but $c^{U(\mathfrak{sl}_2(F))}_1(M_2(F))>1$. 
  \end{remark}

 \begin{theorem}\label{TheoremAssGPIexpEqual}
Let $A$ be a finite dimensional associative algebra
over an algebraically closed field $F$ of characteristic $0$. Suppose a connected reductive affine algebraic group $G$ acts on $A$ rationally by automorphisms and anti-automorphisms.
 Then $\PIexp^G(A)=\PIexp(A)$.
  \end{theorem}
  \begin{proof} 
  We compare~(\ref{EqPIexp}) with the formula for the ordinary PI-exponent~\cite[Section 6.2]{ZaiGia} and apply Lemma~\ref{LemmaGSimple}.
  \end{proof}
  
 \subsection{Lie algebras}\label{SubsectionPIexpLie} Using~\cite[Section~1.8]{ASGordienko5}, we obtain the following formula
 for $\PIexp^{U(\mathfrak g)}(L)$ where $L$ is a finite dimensional Lie algebra over an algebraically
 closed field $F$ of characteristic $0$ with an action of a semisimple Lie algebra $\mathfrak g$
 by derivations. This formula is analogous to the formula for the $\PIexp(L)$ (see~\cite[Definition~2]{ZaiLie}) which was later naturally generalized for~$\PIexp^{G}(L)$ and~$\PIexp^{\mathrm{gr}}(L)$
 in~\cite{ASGordienko2}.
 
 By Theorem~\ref{TheoremDerLevi}, there exists a $\mathfrak g$-invariant maximal semisimple subalgebra $B$
 such that $L = B \oplus R$ (direct sum of $\mathfrak g$-submodules) where $R$ is the solvable radical of $L$.
 Fix such $\mathfrak g$-invariant maximal semisimple subalgebra $B$.
 
 Consider $\mathfrak g$-invariant ideals $I_1, I_2, \ldots, I_r$,
$J_1, J_2, \ldots, J_r$, $r \in \mathbb Z_+$, of the algebra $L$ such that $J_k \subseteq I_k$,
satisfying the conditions
\begin{enumerate}
\item $I_k/J_k$ is an irreducible $(\mathfrak g,L)$-module, i.e. only trivial
subspaces of $I_k/J_k$ are invariant under the $\mathfrak g$-action and the adjoint $L$-action at the same time;
\item for any $\mathfrak g$-invariant $B$-submodules $T_k$
such that $I_k = J_k\oplus T_k$, there exist numbers
$q_i \geqslant 0$ such that $$\bigl[[T_1, \underbrace{L, \ldots, L}_{q_1}], [T_2, \underbrace{L, \ldots, L}_{q_2}], \ldots, [T_r,
 \underbrace{L, \ldots, L}_{q_r}]\bigr] \ne 0.$$
\end{enumerate}

Let $M$ be an $L$-module. Denote by $\Ann M$ its annihilator in $L$.
Then  $$\PIexp^{U(\mathfrak g)}(L) = \max \left(\dim \frac{L}{\Ann(I_1/J_1) \cap \dots \cap \Ann(I_r/J_r)}
\right)$$
where the maximum is found among all $r \in \mathbb Z_+$ and all $I_1, \ldots, I_r$, $J_1, \ldots, J_r$
satisfying Conditions 1--2. 

\section{$S_n$-cocharacters and an upper bound for codimensions}
\label{SectionSncocharAndUpper}

One of the main tools in the investigation of polynomial
identities is provided by the representation theory of symmetric groups.

In this section $H$ is an arbitrary associative algebra with $1$. When we consider $H$-module Lie algebras, we require from $H$ to be a Hopf algebra.

Let $A$ be an associative algebra with a generalized $H$-action
over a field $F$ of characteristic $0$.
 The symmetric group $S_n$  acts
 on the spaces $\frac {P^H_n}{P^H_{n}
  \cap \Id^H(A)}$
  by permuting the variables.
  Irreducible $FS_n$-modules are described by partitions
  $\lambda=(\lambda_1, \ldots, \lambda_s)\vdash n$ and their
  Young diagrams $D_\lambda$.
   The character $\chi^H_n(A)$ of the
  $FS_n$-module $\frac {P^H_n}{P^H_n
   \cap \Id^H(A)}$ is
   called the $n$th
  \textit{cocharacter} of polynomial $H$-identities of $A$.
  Analogously, if $L$ is an $H$-module Lie algebra, $\chi^H_n(L)$ is defined
  as the character of the
  $FS_n$-module $\frac {V^H_n}{V^H_n
   \cap \Id^H(L)}$.
  We can rewrite $\chi^H_n(A)$ as
  a sum $$\chi^H_n(A)=\sum_{\lambda \vdash n}
   m(A, H, \lambda)\chi(\lambda)$$ of
  irreducible characters $\chi(\lambda)$.
Let  $e_{T_{\lambda}}=a_{T_{\lambda}} b_{T_{\lambda}}$
and
$e^{*}_{T_{\lambda}}=b_{T_{\lambda}} a_{T_{\lambda}}$
where
$a_{T_{\lambda}} = \sum_{\pi \in R_{T_\lambda}} \pi$
and
$b_{T_{\lambda}} = \sum_{\sigma \in C_{T_\lambda}}
 (\sign \sigma) \sigma$,
be Young symmetrizers corresponding to a Young tableau~$T_\lambda$.
Then $M(\lambda) = FS e_{T_\lambda} \cong FS e^{*}_{T_\lambda}$
is an irreducible $FS_n$-module corresponding to
 a partition~$\lambda \vdash n$.
  We refer the reader to~\cite{Bahturin, DrenKurs, ZaiGia}
   for an account
  of $S_n$-representations and their applications to polynomial
  identities.

\begin{lemma}\label{LemmaCocharUpper}
Let $A$ be a finite dimensional associative algebra with a generalized $H$-action
or finite dimensional $H$-module Lie algebra
over a field $F$ of characteristic $0$, with an $H$-invariant nilpotent ideal $I \subseteq A$, $I^p=0$ for some $p \in \mathbb N$. Suppose $n\in\mathbb N$ and $\lambda = (\lambda_1, \ldots, \lambda_s) \vdash n$. Then if $\sum_{k=(\dim A)-(\dim I)+1}^s \lambda_k \geqslant p$, we have $m(A, H, \lambda)=0$.
\end{lemma}
\begin{proof}
It is sufficient to prove that $e^{*}_{T_\lambda}f \in \Id^H(A)$ for all $f \in P_n$ and for all Young tableaux $T_\lambda$ corresponding to $\lambda$.

Fix a basis in $A$ that contains a basis of $I$.
 Note that
$e^{*}_{T_\lambda} = b_{T_\lambda} a_{T_\lambda}$
and $b_{T_\lambda}$ alternates the variables of each column
of $T_\lambda$. Hence if we make a substitution and $
e^{*}_{T_\lambda} f$ does not vanish, then this implies that different basis elements
are substituted for the variables of each column.
Therefore, at least $\sum_{k=(\dim A)-(\dim I)+1}^s \lambda_k \geqslant p$ elements must be taken from $I$.
Since $I^p = 0$, we have $e^{*}_{T_\lambda} f \in \Id^H(L)$.
\end{proof}

\begin{theorem}\label{TheoremUpperBoundNilp}
Let $A$ be a finite dimensional associative algebra with a generalized $H$-action
or a finite dimensional $H$-module Lie algebra
over a field $F$ of characteristic $0$, with an $H$-invariant nilpotent ideal $I \subsetneqq A$.
Then there exist $C_3 > 0$ and $r_3\in\mathbb R$ such that
$$c^H_n(A) \leqslant C_3 n^{r_3} ((\dim A)-(\dim I))^n \text{ for all } n\in\mathbb N.$$
\end{theorem}
\begin{proof}
Lemma~\ref{LemmaCocharUpper} and~\cite[Lemmas~6.2.4, 6.2.5]{ZaiGia}
imply
$$
\sum_{m(A,H, \lambda)\ne 0} \dim M(\lambda) \leqslant C_4 n^{r_4} ((\dim A)-(\dim I))^n
$$
for some constants $C_4, r_4 > 0$.

 If $A$ is an $H$-module Lie algebra, $m(A, H, \lambda)$ are polynomially bounded by~\cite[Theorem~4]{ASGordienko5}. If $A$ is an associative algebra with a generalized $H$-action, we can use the same arguments.
This yields the upper bound.
\end{proof}

\section{Examples and criteria for simplicity}
\label{SectionExamples}

In this section, except Subsection~\ref{SubsectionExampleNonSemisimple}, we assume the base field $F$ to be algebraically closed
  of characteristic $0$.
 
 \subsection{Algebras with a (generalized) $H$-action}
  
  We will use the following two facts:
  
    \begin{example}[{\cite[Example 10]{ASGordienko5}}]\label{ExampleHSimpleLie}
 Let $B$ be a finite dimensional semisimple $H$-module Lie algebra where
  $H$ is a Hopf algebra. If $B$ is $H$-simple, then there exist $C > 0$, $r \in \mathbb R$
  such that $$C n^r (\dim B)^n \leqslant c_n^{H}(B)
  \leqslant (\dim B)^{n+1} \text{ for all } n\in\mathbb N.$$
 \end{example}

  \begin{example}[{\cite[Example 11]{ASGordienko5}}]\label{ExampleHSemiSimpleLie}
 Let $L=B_1 \oplus B_2 \oplus \ldots \oplus B_q$
  be a finite dimensional semisimple $H$-module Lie algebra where
  $H$ is a Hopf algebra and $B_i$ are $H$-simple Lie
  algebras. Let $d := \max_{1 \leqslant k
  \leqslant q} \dim B_k$. Then there exist $C_1, C_2 > 0$, $r_1, r_2 \in \mathbb R$
  such that $$C_1 n^{r_1} d^n \leqslant c_n^{H}(L)
  \leqslant C_2 n^{r_2} d^n \text{ for all } n\in\mathbb N.$$
 \end{example}
 
 Theorem~\ref{TheoremHCrSimpleLie} below is a generalization of~\cite[Theorem~15]{ASGordienko5}.
 
 \begin{theorem}\label{TheoremHCrSimpleLie}
Let $L$ be a finite dimensional $H$-module Lie algebra where $H$ is a Hopf algebra. Suppose the nilpotent radical $N$ of $L$ is $H$-invariant. Then $\PIexp^H(L)=\dim L$ if and only if $L$ is an $H$-simple semisimple
algebra.
\end{theorem}
\begin{proof}
If $L$ is $H$-simple semisimple, then $\PIexp^H(L)=\dim L$ by Example~\ref{ExampleHSimpleLie}.
Suppose $\PIexp^H(L)=\dim L$. Then by Theorem~\ref{TheoremUpperBoundNilp}, $N=0$.
By~\cite[Proposition 2.1.7]{GotoGrosshans}, $[L, R] \subseteq N = 0$ where $R$ is the solvable radical of $L$. Hence $R = Z(L)\subseteq N=0$ and $L$ is semisimple. By~\cite[Theorem~6]{ASGordienko4}, $L$ is the sum of $H$-simple Lie algebras. Now we apply Example~\ref{ExampleHSemiSimpleLie}.
\end{proof}

Theorem~\ref{TheoremAssGenHopf} implies the following generalization of~\cite[Example~7]{ASGordienko3}:

  \begin{example}\label{ExampleHGenSemiSimpleAss}
 Let $A=B_1 \oplus B_2 \oplus \ldots \oplus B_q$
  be an associative algebra with a generalized $H$-action,
   where $B_i$ are finite dimensional $H$-simple semisimple
  algebras and $H$ is an associative algebra with~$1$. Let $d := \max_{1 \leqslant k
  \leqslant q} \dim B_k$. Then there exist $C_1, C_2 > 0$, $r_1, r_2 \in \mathbb R$
  such that $$C_1 n^{r_1} d^n \leqslant c_n^{H}(A)
  \leqslant C_2 n^{r_2} d^n \text{ for all } n\in\mathbb N.$$
 \end{example}

Using~\cite[Lemma~4]{ASGordienko3}, we get
  \begin{example}\label{ExampleHGenSimpleAss}
 Let $B$ be an $H$-simple semisimple associative algebra with a generalized $H$-action
   where $H$ is an associative algebra with~$1$. Then there exist $C > 0$, $r \in \mathbb R$
  such that $$C n^r (\dim B)^n \leqslant c_n^{H}(B)
  \leqslant (\dim B)^{n+1} \text{ for all } n\in\mathbb N.$$
 \end{example}
 
  \begin{theorem}\label{TheoremHCrSimpleAss}
Let $A$ be a finite dimensional $H$-module associative algebra where $H$ is a Hopf algebra with a bijective antipode. Suppose the Jacobson radical $J(A)$ is $H$-invariant. Then $\PIexp^H(A)=\dim A$ if and only if $A$ is $H$-simple.
\end{theorem}
\begin{proof}
If $A$ is $H$-simple, then $A$ is semisimple since $J(A)$ is $H$-invariant. Hence $\PIexp^H(A)=\dim A$ by Example~\ref{ExampleHGenSimpleAss}. Suppose $\PIexp^H(A)=\dim A$. Then by Theorem~\ref{TheoremUpperBoundNilp}, $J(A)=0$. Hence $A$ is semisimple. By Lemma~\ref{LemmaHWedderburn}, $A$ is the sum of $H$-simple associative algebras. Now we apply Example~\ref{ExampleHGenSemiSimpleAss}.
\end{proof}

\subsection{Algebras with derivations}

Now we consider the case when $H=U(\mathfrak g)$ for some Lie algebra $\mathfrak g$.
Recall that by~\cite[Chapter~III, Section~6, Theorem~7]{JacobsonLie} 
 the solvable radical and the nilpotent radical of a finite dimensional Lie algebra are invariant under all derivations. By~\cite[Lemma 3.2.2]{Dixmier}, the Jacobson radical (which coincides with the prime radical) of a finite dimensional associative algebra is invariant under all derivations too.
 
\begin{example}\label{ExampleUgSimple}
 Let $B$ be a  simple finite dimensional Lie or associative algebra with an action of a Lie algebra $\mathfrak g$
 by derivations. Then there exist $C > 0$, $r \in \mathbb R$
  such that $$C n^r (\dim B)^n \leqslant c_n^{U(\mathfrak g)}(B)
  \leqslant (\dim B)^{n+1} \text{ for all } n\in\mathbb N.$$
 \end{example}
 \begin{proof}
 We use Examples~\ref{ExampleHSimpleLie} and \ref{ExampleHGenSimpleAss}.
 \end{proof}


   \begin{example}\label{ExampleUgSemiSimple}
 Let $B=B_1 \oplus B_2 \oplus \ldots \oplus B_q$ (direct sum of ideals)
  be a finite dimensional semisimple Lie or associative algebra with an action
  of a Lie algebra $\mathfrak g$ by derivations,
  where $B_i$ are simple
  algebras. Let $d := \max_{1 \leqslant k
  \leqslant q} \dim B_k$. Then there exist $C_1, C_2 > 0$, $r_1, r_2 \in \mathbb R$
  such that $$C_1 n^{r_1} d^n \leqslant c_n^{U(\mathfrak g)}(B)
  \leqslant C_2 n^{r_2} d^n \text{ for all } n\in\mathbb N.$$
 \end{example}
 \begin{proof} By Lemma~\ref{LemmaDerSimpleSum}, $B_i$ are $\mathfrak g$-invariant.
 Now we use Examples~\ref{ExampleHSemiSimpleLie} and \ref{ExampleHGenSemiSimpleAss}.
 \end{proof}

Finally, we obtain a criterion for (differential) simplicity in terms of differential PI-exponent:

\begin{theorem}\label{TheoremUgCrSimple}
Let $A$ be a finite dimensional Lie or associative algebra
 with an action of a Lie algebra $\mathfrak g$
 by derivations. Then $\PIexp^{U(\mathfrak g)}(A)=\dim A$ if and only if $A$ is $\mathfrak g$-simple if and only if $A$ is simple.
\end{theorem}
\begin{proof}
We use Theorems~\ref{TheoremHCrSimpleLie}, \ref{TheoremHCrSimpleAss}, and Lemma~\ref{LemmaDerSimple}.
\end{proof}

\subsection{$G$-algebras}
 
  If a group is acting on an algebra by automorphisms and anti-automorphisms, the radicals are invariant under this action. In the case of Lie algebras every anti-automorphism is a negative automorphism, so we can always
  restrict ourselves to the case when a group is acting on a Lie algebra by automorphisms only.
  (See~\cite[Lemma~28]{ASGordienko5}.)

     \begin{example}\label{ExampleGSimple}
 Let $B$ be a finite dimensional Lie or associative algebra with an action of a group $G$ by automorphisms and anti-automorphisms. If $B$ is $G$-simple, then there exist $C > 0$, $r \in \mathbb R$
  such that $$C n^r (\dim B)^n \leqslant c_n^{G}(B)
  \leqslant (\dim B)^{n+1} \text{ for all } n\in\mathbb N.$$
 \end{example}
  \begin{proof}
 We use Examples~\ref{ExampleHSimpleLie} and \ref{ExampleHGenSimpleAss}.
 \end{proof}
 
      \begin{example}\label{ExampleGSemiSimple}
  Let $B=B_1 \oplus B_2 \oplus \ldots \oplus B_q$ (direct sum of $G$-invariant ideals)
  be a finite dimensional semisimple associative algebra with an action
  of a of a group $G$ by automorphisms and anti-automorphisms,
  where $B_i$ are $G$-simple algebras. Let $d := \max_{1 \leqslant k
  \leqslant q} \dim B_k$. Then there exist $C_1, C_2 > 0$, $r_1, r_2 \in \mathbb R$
  such that $$C_1 n^{r_1} d^n \leqslant c_n^{G}(B)
  \leqslant C_2 n^{r_2} d^n \text{ for all } n\in\mathbb N.$$
 \end{example}
 \begin{proof}
 We use Examples~\ref{ExampleHSemiSimpleLie} and \ref{ExampleHGenSemiSimpleAss}.
 \end{proof}

  Now we obtain a criterion for $G$-simplicity:
  
  \begin{theorem}\label{TheoremGCrSimple}
Let $A$ be a finite dimensional Lie or associative algebra
with an action of a group $G$ by automorphisms and anti-automorphisms. Then $\PIexp^{G}(A)=\dim A$ if and only if $A$ is $G$-simple.
\end{theorem}
\begin{proof} 
We use Theorems~\ref{TheoremHCrSimpleLie} and \ref{TheoremHCrSimpleAss}.
\end{proof}
  
\subsection{Graded algebras}  
  
   Using Lemma~\ref{LemmaGrToHatG},
    we obtain the following examples and criterion for graded simplicity:
    
      \begin{example}\label{ExampleGrSimple}
 Let $B$ be a finite dimensional associative algebra graded by an Abelian group $G$. If $B$ is graded simple, then there exist $C > 0$, $r \in \mathbb R$
  such that $$C n^r (\dim B)^n \leqslant c_n^{\mathrm{gr}}(B)
  \leqslant (\dim B)^{n+1} \text{ for all } n\in\mathbb N.$$
 \end{example}

      \begin{example}\label{ExampleGrSemiSimple}
  Let $B=B_1 \oplus B_2 \oplus \ldots \oplus B_q$ (direct sum of graded ideals)
  be a finite dimensional semisimple Lie or associative algebra graded by an Abelian group $G$,
  where $B_i$ are graded simple algebras. Let $d := \max_{1 \leqslant k
  \leqslant q} \dim B_k$. Then there exist $C_1, C_2 > 0$, $r_1, r_2 \in \mathbb R$
  such that $$C_1 n^{r_1} d^n \leqslant c_n^{\mathrm{gr}}(B)
  \leqslant C_2 n^{r_2} d^n \text{ for all } n\in\mathbb N.$$
 \end{example}

\begin{theorem}\label{TheoremGrCrSimple}
Let $A$ be a finite dimensional associative algebra
 graded by an Abelian group. Then $\PIexp^{\mathrm{gr}}(A)=\dim A$ if and only $A$ is graded simple.
\end{theorem}

When the grading group $G$ is finite, we can use~\cite[Lemma~1]{ASGordienko3}
and derive the above from~Examples~\ref{ExampleHGenSemiSimpleAss}, \ref{ExampleHGenSimpleAss},
and Theorem~\ref{TheoremHCrSimpleAss}, even if $G$ is not Abelian.

Analogous examples and criterion for Lie algebras were obtained in~\cite{ASGordienko5}.

\subsection{Examples of non-semisimple algebras}\label{SubsectionExampleNonSemisimple}
 
We conclude the section with the following two examples:

\begin{example}\label{ExampleMmMm} Let $F$ be a field of characteristic $0$.
Consider the associative subalgebra $$A = \left\lbrace\left.\left(\begin{array}{cc} C & D \\ 0 & 0 \end{array}\right)
\right| C,D \in M_m(F)\right\rbrace \subset M_{2m}(F) \text{ where }m\geqslant 2.$$
Define the linear embedding $\varphi \colon \mathfrak{sl}_m(F) \hookrightarrow A$,
$\varphi(C)=\left(\begin{array}{cc} C & 0 \\ 0 & 0 \end{array}\right)$
and the following $\mathfrak{sl}_m(F)$-action on $A$ by derivations:
$a \cdot b = [\varphi(a), b]$ for all $a \in \mathfrak{sl}_m(F)$
and $b \in A$.
Then there exist $C_1, C_2 > 0$, $r_1, r_2 \in \mathbb R$
  such that $$C_1 n^{r_1} m^{2n} \leqslant c_n^{U(\mathfrak{sl}_m(F))}(A)
  \leqslant C_2 n^{r_2} m^{2n} \text{ for all } n\in\mathbb N.$$
\end{example}
\begin{proof} As we mentioned in the proof of Theorems~\ref{TheoremDer} and~\ref{TheoremDerSum} (Subsection~\ref{SubsectionGApplToDiff}), differential codimensions do not change upon
an extension of the base field. Hence we may assume $F$ to be algebraically closed.

Note that $A = B \oplus J$ (direct sum of $\mathfrak{sl}_m(F)$-submodules)
where $$B=\left\lbrace\left.\left(\begin{array}{cc} C & 0 \\ 0 & 0 \end{array}\right)
\right| C \in M_m(F)\right\rbrace$$ is a maximal semisimple subalgebra
(which is simple) and $$J = \left\lbrace\left.\left(\begin{array}{cc} 0 & D \\ 0 & 0 \end{array}\right)
\right| D \in M_m(F)\right\rbrace$$ is the Jacobson radical of $A$.
Hence~(\ref{EqPIexp}) implies the claimed asymptotics.
\end{proof}

\begin{example}\label{ExampleslmMm} Let $F$ be a field of characteristic $0$.
Consider the Lie subalgebra $$L = \left\lbrace\left.\left(\begin{array}{cc} C & D \\ 0 & 0 \end{array}\right)
\right| C\in\mathfrak{sl}_m(F),\ D \in M_m(F)\right\rbrace \subset \mathfrak{sl}_{2m}(F) \text{ where }m\geqslant 2.$$
Define the linear embedding $\varphi \colon \mathfrak{sl}_m(F) \hookrightarrow L$,
$\varphi(C)=\left(\begin{array}{cc} C & 0 \\ 0 & 0 \end{array}\right)$
and the following $\mathfrak{sl}_m(F)$-action on $L$ by derivations:
$a \cdot b = [\varphi(a), b]$ for all $a \in \mathfrak{sl}_m(F)$
and $b \in L$.
Then there exist $C_1, C_2 > 0$, $r_1, r_2 \in \mathbb R$
  such that $$C_1 n^{r_1} (m^2-1)^n \leqslant c_n^{U(\mathfrak{sl}_m(F))}(L)
  \leqslant C_2 n^{r_2} (m^2-1)^n \text{ for all } n\in\mathbb N.$$
\end{example}
\begin{proof} Again, differential codimensions do not change upon
an extension of the base field. Hence we may assume $F$ to be algebraically closed.

Note that $A = B \oplus R$ (direct sum of $\mathfrak{sl}_m(F)$-submodules)
where $$B=\left\lbrace\left.\left(\begin{array}{cc} C & 0 \\ 0 & 0 \end{array}\right)
\right| C \in \mathfrak{sl}_m(F)\right\rbrace$$ is a maximal semisimple subalgebra
(which is simple) and $$R = \left\lbrace\left.\left(\begin{array}{cc} 0 & D \\ 0 & 0 \end{array}\right)
\right| D \in M_m(F)\right\rbrace$$ is the solvable (and nilpotent) radical of $L$.
Then if $I_1, \ldots, I_r$, $J_1, \ldots, J_r$ satisfy Conditions 1--2 from
Subsection~\ref{SubsectionPIexpLie}, we have $R \subseteq \Ann(I_1/J_1) \cap \dots \cap \Ann(I_r/J_r)$,
since $R$ is a nilpotent ideal. Thus $\PIexp^{U(\mathfrak{sl}_m(F))}(L) \leqslant (\dim L)-(\dim R)=m^2-1$.
However, $L/R \cong B$ is a simple $B$-module. Hence $I_1 = L$ and $J_1 = R$ satisfy Conditions 1--2.
Now we notice that $\Ann(L/R)=R$, $\dim(L/\Ann(L/R))=m^2-1$,
and Theorem~\ref{TheoremDer} implies the claimed asymptotics.
\end{proof}

In both examples, the differential PI-exponent coincides with the ordinary one.

\section*{Acknowledgements}

This work started while the first author was an AARMS postdoctoral fellow at Memorial University of Newfoundland, whose faculty and staff he would like to thank for hospitality and many useful discussions. Both authors are grateful to Yuri Bahturin who suggested that they study polynomial $H$-identities.

\end{document}